\documentclass[11pt, reqno]{amsart}
\usepackage{amssymb}
\usepackage{verbatim}

\usepackage{amsmath,amsfonts,epsfig, graphics, graphicx, amssymb}
\usepackage{pstricks, pst-plot, pst-node, pst-tree, pstricks-add}
\newtheorem{theorem}{Theorem}[section] 
\newtheorem{lemma}[theorem]{Lemma}

\newcommand{\fl}[1]{\left\lfloor #1\right\rfloor}
\newcommand{\ceil}[1]{\left\lceil #1\right\rceil}

\newcommand{\field}[1]{\mathbb{#1}}

\newcommand{\abs}[1]{\left\vert #1 \right\vert}
\newcommand{\inv}[1]{\frac{1}{#1}}

\newcommand{\ssy}{\mu}

\newcommand{\tns}{T_n}
\newcommand{\tnss}{T_n^{\ast}}
\newcommand{\mone}{\nu_{n,p}}
\newcommand{\monek}{\nu_{n,{\mathfrak p}}}

\newcommand{\moneTwok}{\nu_{2,\fp}}

\newcommand{\moneThreek}{\nu_{3,\fp}}

\newcommand{\mpik}{P_{n,{\mathfrak p}}(\ssy)}
\newcommand{\mpikk}{P_{n,{\mathfrak p}}(\ssy; K_{\mathfrak p})}

\newcommand{\mpist}{P_{n,p}^{\ast}(\ssy)}
\newcommand{\mpistk}{P_{n,\fp}^{\ast}(\ssy)}

\newcommand{\pgpk}{P_{n,{\mathfrak p}}}

\newcommand{\pgpsk}{P_{n,\fp}^{\ast}}

\newcommand{\pgpTwok}{P_{2,\fp}}

\newcommand{\pgpThreek}{P_{3,\fp}}

\newcommand{\haark}{h_{\mathfrak p}}
\newcommand{\nssy}{\nu_{n}(\ssy)}

\newcommand{\tv}{\tilde{\nu}}

\newcommand{\fp}{{\mathfrak{p}}}

\newcommand{\sO}{{\mathfrak O}}

\newcommand \niq{{M(q; i)}}

\newcommand \nmq{{M(q; m)}}

\newcommand \nnp{{M(p; n)}}

\def\F{{\field F}}

\def\Z{\field{Z}}
\def\Q{\field{Q}}

\providecommand{\Gal}{\mathop{\rm Gal}\nolimits}%

\title[Probabilistic Galois Theory over $P$-adic Fields]
 {Probabilistic Galois Theory over $P$-adic Fields} 
 \author{Benjamin L. Weiss}
 \address{Dept. of Mathematics, Technion-Israel Institute of Technology, Haifa 32000, Israel.\newline\newline
 Current Address: Department of Mathematics \& Statistics \\
University of Maine\\
Orono, ME 04469\\}
   \email{benjamin.weiss@maine.edu}


\thanks{The author  acknowledges
support from NSF Grant DMS-0801029  from
 a Mathematics Research Assistantship Fellowship at the University of Michigan and from ISF grant \# 776/09.
This is a version of a paper that appeared in the May 2013 issue of the Journal of Number Theory.}

\date{Sept. 1, 2014}

\begin{document}
\maketitle

\begin{abstract}
We estimate several probability distributions arising from the study of random, monic polynomials
of degree $n$ with coefficients in the integers of a general $p$-adic field $K_{\fp}$ having
residue field with $q= p^f$ elements.  We estimate the
distribution of the degrees of irreducible factors of the polynomials, with
 tight error bounds valid when $q> n^2+n$.
We also estimate
the distribution of Galois groups of such polynomials, showing that for fixed $n$,
almost all Galois groups are cyclic in the limit $q \to \infty$. In particular, we show that the 
Galois groups are cyclic with probability at least $1 - \inv{q}$.
We obtain exact formulas in the case of $K_\fp$ for all $p > n$ when $n=2$ and $n=3$.
\end{abstract}

\section{Introduction}\label{s:intro}

    The study of the Galois groups of  polynomials has a long history.
    In  1892 Hilbert \cite{Hilbert:1892} showed that for an irreducible  polynomial $f(X,t) \in (\Q[X])[t]$ there 
    are infinitely many rational numbers $x$ for which $f(x,t)$ is irreducible in $\Q[t]$. 
      In 1936
    van der Waerden \cite{Waerden:1936} gave a quantitative form of this assertion. Consider 
       the set of   degree $n$ monic  polynomials with integer coefficients restricted to a box
    $|a_i| \le B$. Van der Waerden showed that a polynomial drawn at random from this set has Galois
    group $S_n$ with probability going to $1$
    as $B \to \infty$. 
       (See Section \ref{sec13} for more details). 
       The study of the  distribution of Galois groups for polynomials with 
    coefficients drawn  from some given probability distribution is now termed probabilistic Galois theory.

     In this paper we address problems analogous to that of van der Waerden for $p$-adic fields. 
   We investigate the distribution of Galois groups of monic polynomials of fixed degree $n$ with
    coefficients drawn from
 the integers $\sO_{\fp}$ of a general $p$-adic field $K_{\fp}$, 
 i.e. a finite extension of $\Q_p$.
        Quantitative  questions parallel to van der Waerden's study
   concern, firstly,  the probability a randomly drawn $p$-adic polynomial is irreducible, secondly, the 
   probability such a polynomial has irreducible factors of given degrees, and thirdly,  the distribution of Galois groups
    of such polynomials.
    One can ask about the probability of occurrence of each kind of Galois group
    for  such polynomials. In this paper we will find  bounds concerning such probabilities,
    and exact formulas for degrees $n=2$ and $n=3$.

    Our main results are:\medskip

(1) Estimates for the probability distribution of splitting types of monic degree $n$ polynomials
with coefficients drawn from the  integers $\sO_{\fp}$ in a $p$-adic field $K_{\fp}$.  These include  
exact formulas if one conditions the polynomials to
have discriminant relatively prime to $p$. (Theorems \ref{t:uncond}
and  \ref{t:uncond2});\medskip

(2) Estimates for the probability  distribution of Galois groups of monic degree $n$ polynomials  with coefficients
drawn from the  integers $\sO_{\fp}$ in a $p$-adic field $K_{\fp}$ (Theorem \ref{c:uncond});\medskip

(3) The existence of a limiting distribution for the distributions
of the splitting types, resp. the  Galois groups, as $q \to \infty$, where $q= p^{f_{\fp}}$
is the order of the residue class field of $K_{\fp}$ (Theorems \ref{t:uncond2cor} and  \ref{c:uncondcor});\medskip

(4) Exact probability calculations for (1), (2) above for  the case of
 any $p$-adic field $K_\fp$ for degrees $n=2, 3$ and all primes $p>n$ (see Section \ref{s:examp}).\medskip

 Below we state the results in more detail.
 These results reveal both some interesting parallels and some sharp differences with the distributions for
  polynomials with integer coefficients; we give a comparison in Section \ref{sec13}. 

\subsection{Factorization of $p$-adic Polynomials}

We let $K_{\fp}$ denote a finite extension of the $p$-adic field $\Q_p$, and we let $\sO_{\fp}$ denote the integers
in the field $K_{\fp}.$ Two invariants of $K_{\fp}$ are its {\em ramification index} $e= e_{\fp}$ 
and its residue class field degree $f= f_{\fp}$. These are 
given by
$
(p)\sO_{\fp} = (\pi^e) \sO_{\fp},
$
where $\pi= \pi_{\fp}$ is a uniformizing parameter for the maximal ideal $\fp :=\pi \sO_{p}$, and 
the residue class field given by $\F_q= \sO_{\fp}/ \pi \sO_{\fp}$ with $q= p^{f_\fp}.$
 We normalize the additive $\fp$-adic Haar
measure $\haark$ on $K_{\fp}$ to assign mass $1$ to the $\fp$-adic integers $\sO_{\fp}.$

We now consider monic, degree $n$ polynomials 
$f(x) = x^n + a_{n-1}x^{n-1} + \dots + a_0$ with $(a_0, \ldots, a_{n-1}) \in (\sO_{\fp})^n.$
For fixed $p$-adic field $K_\fp$, let $\haark$ be the Haar measure of the $\fp$-adic integers normalized to be a probability measure. For fixed positive integer $n$ and $p$-adic field $K_\fp$, we take the probability distribution of the $a_i$ as independent and distributed as $\haark$ and denote the product measure $\monek$.

Given a positive integer $n$, any unique factorization domain $R$, and any monic degree $n$ polynomial $f(x) \in R[x]$, we can factor $f(x)$ uniquely as $f(x) = \prod_{i=1}^kg_i(x)^{e_i}$, where the $e_i$ are positive integers and the $g_i(x)$ are distinct, monic, irreducible, and non-constant. 
Following the framework of Bhargava \cite{Bhargava:2007}, 
we  define the \emph{splitting type} of such a polynomial to be the formal symbol
\[
\ssy(f) := \left( \deg(g_1)^{e_1},\ \deg(g_2)^{e_2},\ \ldots\ \deg(g_k)^{e_k} \right)
\]
 where $k$ is the number of distinct irreducible factors of $f(x)$.
 Here we order the degrees in decreasing order.
We then define $\tns$ to be the set
of all possible splitting types above. Thus $T_3 =\{ (111), (21), (3), (1^2 1), (1^3)\}$.
Next we 
 define $\tnss$ to be the set of splitting types such that all the exponents are equal to $1$.
 Thus $T_3^{\ast}= \{ (111), (21), (3)\}.$ The latter sets are  the ones relevant to this paper.

  The class $\tnss$  labels  the possible 
  splitting types of square-free polynomials.  Each element $\ssy=\left(\ssy_1,\ \ssy_2,\ \ldots\ \ssy_k\right) \in \tnss$
  with $\ssy_1 \ge \ssy_2 \ge... \ge \ssy_k$ describes  a partition of
 $n$,  has
 associated  a unique conjugacy class  $C_\ssy \subset S_n$, the symmetric group on $n$ elements. The conjugacy class is the set of all elements of $S_n$ whose cycle lengths are equal to the  numbers
 $\ssy_1, \ldots, \ssy_k$ ordered in nonincreasing order.

For a splitting type $\ssy$,  let
$c_i(\ssy)$ count  the number of terms
 $\ssy_j$ which are equal to $i$.

We let $\mpik= \mpikk$ denote the set of
tuples $(a_0, a_1, ..., a_{n-1}) \in \left(\sO_\fp\right)^n$ for which the polynomial $f(x) = x^n+a_{n-1}x^{n-1} + \cdots + a_0$
 has  splitting type $\ssy \in \tnss$, so in particular $f(x)$ is square-free. 
 This set is a full measure subset of $\left(\sO_\fp\right)^n$. We will also write $f(x) \in \mpik$
  if the tuple of coefficients of $f(x)$ are in $\mpik$.

 We next  let $\mpistk$ be the subset of $\mpik$ whose associated discriminants are not divisible by $\fp$, i.e.\ their associated polynomials have modulo $\fp$ reductions which are square-free.   Theorem ~\ref{t:uncond} below determines the measure of 
 $\mpistk$.
 To state this result, we  also use the number $\nmq$ with $q= p^{f}$ which
 counts the number of monic, irreducible polynomials of degree $m$ in $\F_q[x]$.
By a  formula of Gauss this number is
 \[
  \nmq:=\inv{m}\sum_{d|m}\mu(\frac{m}{d})q^d = \frac{1}{m} \sum_{d|m}\mu(d)q^{\frac{m}{d}}.
 \]
  Thus  $M(q; 1)= q, M(q; 2)= \frac{1}{2} (q^2 -q)$, etc.
 For any positive integer $n$ and $p$-adic extension $K_{\fp}$ we will denote
 \[
 \pgpsk:= \bigcup_{\ssy \in \tns} \mpistk,
 \]
  which corresponds to the set of monic degree $n$ polynomials $f(x) \in \sO_\fp[x]$ whose discriminant is not divisible by $\fp$. In the following result we
 adopt the convention that ${A\choose 0} = 1$ for any non-negative integer $A$, and ${A\choose B} = 0$ whenever $B > A$.

\begin{theorem}\label{t:uncond}
Fix any integer $n \ge 2$ and prime $p$, and let $K_{\fp}$ be a finite extension of $\Q_p$, with
residue class degree $q=p^f$.
 The probability that a random monic, degree $n$ polynomial $f(x) \in \sO_{\fp}[x]$ has coefficients belonging to $\pgpsk$ is
\[\monek(\pgpsk)= 1- \frac{1}{q}.\]
For any splitting type $\ssy \in \tnss$, we have
\[\monek(\mpistk) = \inv{q^{n}}\prod_{i=1}^n{\niq \choose c_i(\ssy)}. \]
 \end{theorem}

This result is based on  the fact that if the modulo $\fp$ reduction of a
monic  polynomial in $\sO_\fp[x]$ is square-free,
then the splitting type of the polynomial is the same as the splitting type of its modulo $\fp$ reduction.
The probability  $1-\frac{1}{q}$ of a monic polynomial over $\F_q[X]$ having  square-free factorization was
noted by Zieve (see Lemma ~\ref{l:sqfree}).
The formula for the number of square-free polynomials over $\F_q$
having  a given $\ssy$ was noted by Cohen \cite[p.256]{Cohen:1970}.
The formula is valid  for all $n$ and $q$, including those $q \le n$.

For any element $\ssy = \left(\ssy_1,\ \ssy_2,\ \ldots\ \ssy_k\right)  \in \tnss$, let $\nssy$ be the proportion of elements in $S_n$ which are in
the conjugacy class  $C_\ssy$.  This is the probability distribution
on $\tnss$ given by
\begin{equation}\label{nu-n}
\nssy := \frac{|C_{\ssy}|}{|S_n|} = \frac{ |C_{\ssy}|}{n!}.
\end{equation}
A  well known explicit formula for  $\nssy$ in terms of $\ssy_1,\ldots, \ssy_k$ is
given in Lemma \ref{d:cycprob}. 
The distribution $\nssy$ is  exactly the distribution associated with the Chebotarev density theorem
for Artin symbols in $S_n$-extensions of $\Q$, compare \cite[pp. 410-411]{Lagarias:1977}.

Next, we obtain bounds on these probabilities when $q$ is sufficiently large 
compared to $n$, as follows. 

\begin{theorem}\label{t:uncond2}
Fix any positive integer $n \ge 2$ and any splitting type $\ssy \in \tnss$.
Consider monic polynomials with coefficients in
the ring of integers $\sO_{\fp}$ of a finite extension field $K_{\fp}$ over $\Q_p$,
with residue class field $\F_q$ with $q= p^f$.

(1)  If $n \ge 3$ and  $q \ge n^2$ then we have
 \[\abs{\monek(\mpistk) - \nssy} < \frac{n^2}{q-n^2}\nssy.\]

(2)   If $n = 2$ and  $q \ge n^2 + n$ we have
 \[\abs{\monek(\mpistk) - \nssy} < \frac{n^2 +n}{q -( n^2 + n)}\nssy.\]
\end{theorem}

In Theorem \ref{t:uncond}  the probabilities $\monek(\mpistk)$  depend only of the parameter $q$ 
and the splitting type $\ssy$, as do  the  bounds obtained in Theorem \ref{t:uncond2} on these distributions.

As an example, take $K_{\fp} = \Q_p$, and $\ssy= (n)$, which is the splitting type of a polynomial $f(x)$ which
is irreducible over $\Q_p$.
Theorem \ref{t:uncond} asserts in this case that
$\mone(\mpist) = \frac{\nnp}{p^n}$. In this case $\nssy= \frac{1}{n}$
and Theorem \ref{t:uncond2} says that on the range $ n^2+n < p < \infty$ one has
$|\mone(\mpist)- \frac{1}{n}| = O(\frac{1}{p})$,
with the implied constant in the $O$-symbol depending on $n$. In fact,
the formula for $\nnp$
yields the stronger convergence rate
$|\mone(\mpist) -  \frac{1}{n}| = O( p^{-\frac{n}{2}})$.

Another example with $K_{\fp} = \Q_p$  is $\ssy=(1, 1,.., 1)$ ($n$ times), which is the event that
the polynomial splits completely into distinct linear factors over $\Q_p$. In this
case $\nssy = \frac{1}{n!}$. Now $\mone(\mpist) = \frac{1}{p^n}{p\choose n}$
and one can see (by noting that $\frac{p(p-1)\cdots(p-n+1)}{p^n} = 1 + O(\inv{p})$) that $\frac{a}{p} > |\mone((1,1,...,1))- \frac{1}{n!}| > \frac{b}{p}$ for some constants $a$ and $b$. Thus
 in this case the $O(\frac{1}{p})$  estimate of  Theorem \ref{t:uncond2} is the
correct order of magnitude of the error as $p \to \infty$.

Theorem \ref{t:uncond2cor} immediately gives a convergence result for the distribution as $q= p^{f} \to \infty$.

\begin{theorem}\label{t:uncond2cor}
Fix any positive integer $n \ge 2$ and any $\ssy \in \tnss$.
Consider monic polynomials with coefficients in
the ring of integers $\sO_{\fp}$ of a finite extension field $K_{\fp}$ over $\Q_p$,
with residue class field $\F_q$ with $q= p^f$.
Then letting the $p$-adic field  $K_{\fp}$ vary, with either or both of $p$ and $f_{\fp}$
changing in  
any  way  such that $q \to \infty$, one has
 \[   \monek(\mpistk)  \rightarrow \nssy\mbox{ as } q \rightarrow \infty.\]
 In consequence
 \[   \monek(\mpik)  \rightarrow \nssy\mbox{ as } q \rightarrow \infty.\]
\end{theorem}

Two extreme cases of this result are the limit $q= p^f$ with $p$ fixed and $f \to \infty$,
and the limit $q=p$ with $p$ varying and $p \to \infty.$

\subsection{Galois groups of $p$-adic Polynomials}

We next consider
the distribution of Galois groups of random $p$-adic polynomials.
Every polynomial $f(x)$ in $\sO_{\fp}[x]$ has an associated splitting field $K_f$ over $K_{\fp}$ and Galois group 
$G_f =\Gal(K_f/K_{\fp}).$ It is well known that the group $G_f$ is solvable (see Fresenko \cite[p. 59]{Fesenko:2002}).
Additionally, we can realize $G_f$ as a subgroup of the symmetric group $S_n$ by choosing an ordering of the roots of $f$, and identifying an element of $G_f$ with its permutation action on the roots. The embedding of $G_f$ into $S_n$ is only unique up to conjugation, because of the explicit choice of order of the roots of $f$. For any subgroup $G \subset S_n$, let
$C_G$ denote the collection of subgroups of $S_n$ which are conjugate to $G$, and we denote $\abs{C_G}$ the number of subgroups of $S_n$ conjugate to $G$. Let $\pgpk(G)$ be the set of tuples $(a_0, \ldots, a_{n-1}) \in (\sO_{\fp})^n$ for which the polynomial $f(x) = x^n + a_{n-1}x^{n-1} + \cdots + a_0$
has splitting
field with  Galois group in $C_G$; we again will write $f(x) \in \pgpk(G)$ in this case.
Likewise, let $\pgpsk(G)$ be the subset of $\pgpk(G)$ where the discriminants 
of the associated polynomials are not divisible by $\fp$.

We now define a probability distribution $\tv_n(G)$
on conjugacy classes of subgroups of $S_n$ as follows.
For any cyclic subgroup $G\subset S_n$, there is a splitting type $\ssy_G \in \tnss$ 
such that for any $\sigma \in C_{\ssy_G} \subset S_n$ the cyclic group $\langle\sigma\rangle$ 
is conjugate to $G$. We therefore define
\[
\tv_n(G):= \nu_n(\ssy_G).
\]
For any non-cyclic group $G \subset S_n$ we set $\tv_n(G)=0$. We term this distribution the
{\em Erd\H{o}s-Turan distribution} associated to $S_n$, because the  distributions $\tv_n$ were  studied  for their  own sake
in a series of papers by Erd{\H{o}}s and Tur{\'a}n (\cite{Erdos:1965}, \cite{Erdos:1967a} \cite{Erdos:1967}, \cite{Erdos:1968}, \cite{Erdos:1970}, \cite{Erdos:1971}, \cite{Erdos:1972}) over  the period 1965-1972.
Erd\H{o}s and Turan  studied  properties of the order of elements of $S_n$.
 In particular they studied the distributions of the size of the order;
the number of prime divisors of the order; the size of the maximal prime divisor of the order, and related quantities.


\begin{theorem}\label{c:uncond}
Fix a positive integer $n \ge 2$ and let $G$ be a subgroup of $S_n$,
defined up to conjugacy in $S_n$. Consider monic polynomials with coefficients in
the ring of integers $\sO_{\fp}$ of a finite extension field $K_{\fp}$ over $\Q_p$,
with residue class field $\F_q$ with $q= p^f$. 

(1) If $G$ is cyclic, $n \ge 3$, and
 $q \ge n^2$ then we have
\[
 \abs{\monek(\pgpsk(G)) - \tv_{n}(G)} < \frac{n^2}{q - n^2} \, \tv_n(G).
 \]
 In consequence
 \[
 \abs{\monek(\pgpk(G)) - \tv_{n}(G)} <\frac{n^2}{q - n^2}\, \tv_n(G) +\inv{q}.
 \]

 If $G$ is cyclic, $n = 2$, and $q \ge n^2 + n$ then we have
 \[
 \abs{\monek(\pgpsk(G)) - \tv_{n}(G)} <  \frac{ n^2 + n}{q -( n^2 + n)}\, \tv_n(G).
 \]
In consequence
 \[
 \abs{\monek(\pgpk(G)) - \tv_{n}(G)} <  \frac{ n^2 +n}{ q -(n^2 + n)}\, \tv_n(G) +\inv{q}.
 \]

 (2) For non-cyclic $G$ we have:
 \[
 \sum_{\substack{G \subset S_n\\ G  ~\mbox{non-cyclic}}} \frac{\monek(\pgpk(G))}{\abs{C_G}} \le \frac{1}{q}.
 \]
\end{theorem}

In Theorem \ref{c:uncond}  the values $\monek(\pgpk(G))$  might a priori depend  on
 the particular field $K_{\fp}$ and on 
the splitting type $\ssy$, but the  bounds given above depend only on $q$ and $\ssy$. 
 We also  note 
 that  if $G$ is a non-cyclic group then $\monek(\pgpsk(G)) = 0$ because a monic, 
 $\fp$-adic polynomial whose modulo $\fp$ reduction is square-free necessarily has cyclic Galois group.

As an example of these distributions, consider for  $K_{\fp} = \Q_p$,  
the subgroups of $S_4$ given by $G_i= \langle \sigma_i \rangle$ for $i=1,2$ with $\sigma_1=(12)$ and
$\sigma_2= (12)(34)$. Here $G_1$ and $G_2$ are both cyclic of order $2$, but are not conjugate subgroups inside $S_4$.
We have $\tv_{4}(G_1)=\frac{6}{24}$ while $\tv_{4}(G_2)= \frac{3}{24}.$ If one instead  considers the abstract group $\Z/2 \Z$
up to isomorphism type, then we may set   $\tv_{4}(\Z/2\Z)= \frac{9}{24}$, defining it to be  the proportion of elements
having order $2$ in $S_4$.

Theorem \ref{c:uncond} has an immediate consequence, concerning the approach to
a limiting distribution as the field $K_{\fp}$ varies in any way such that $q= p^f \to \infty$.

\begin{theorem}\label{c:uncondcor}
Fix a positive integer $n \ge 2$ and let $G$ be a subgroup of $S_n$,
defined up to conjugacy in $S_n$. 
Consider monic polynomials with coefficients in
the ring of integers $\sO_{\fp}$ of a finite extension field $K_{\fp}$ over $\Q_p$,
with residue class field $\F_q$ with $q= p^f$.
Then letting the $p$-adic field  $K_{\fp}$ vary, with either or both of $p$ and $f_{\fp}$
changing in  
any  way  such that $q \to \infty$, the distributions converge to the Erd\H{o}s-Turan
distribution, i.e.
\[
\monek(\pgpsk(G)) \rightarrow \tv_{n}(G) ~\mbox{ as } ~q \rightarrow \infty.
\]
\end{theorem}

In particular the probability  of a non-cyclic Galois group occurring becomes $0$ in this limit.
For some $p$-adic fields, non-cyclic Galois groups occur with positive probability as soon as $n \ge 3$, but they
do not contribute to the limiting probabilities as $q \to \infty.$  
When $n = 3$,  the density of non-cyclic groups can be bounded above by ${2}{q^{-2}}$.
 When $n > 3$,  we expect there to be a constant $d_n > 0$ so that ${d_n}{q^{-1}}$
  is a lower bound for the density of non-cyclic groups.

To shed some information on these probabilities, 
in Section \ref{s:examp} we explicitly calculate for 
the  probabilities  $\monek(\pgpk(G))$ for $n \in \{2,3\}$, any $p$-adic field $K_\fp$ with prime $p > n$, 
and $G$ any subgroup of $S_n$.
Note that for $n=2,3$ the isomorphism type and conjugacy type of all subgroups of $S_n$ coincide.
The calculations in Section \ref{s:examp} reveal that for $n=2, 3$ the following hold :

(1) The splitting types of polynomials of
degree $n=2,3$ are rational functions of $q=p^{f_p}$ for $p>n$; 

 (2) The probabilities of Galois groups are not
always rational functions of $q$. 
For example, 
$S_3$-extensions of $\Q_p$ never occur for $p \equiv 1~ (\bmod~3),$ but  occur
with positive probability for each $p \equiv 2~ (\bmod~3)$, with $p > 3$
($S_3$-extensions occur for $p=2$ as well, by an argument not given here).

\subsection{Comparison with polynomials over $\Z$: Galois groups}\label{sec13}

We compare and contrast  the results above
concerning Galois groups with known results on
the distribution of
Galois groups of random polynomials with integer coefficients.

In 1936,  van der Waerden \cite{Waerden:1936} showed
 that, for a suitable notion of ``random," a random monic degree $n$ polynomial in $\Z[x]$ has Galois group $S_n$ (and hence is irreducible) with $100\%$ probability.

%
 \begin{theorem}[van der Waerden]\label{t:gal}

  For integer $B \ge 1$  let
 $\mathcal{F}_{n,B}$ be the set of monic, degree $n$ polynomials in $\Z[x]$ with all coefficients in
 the box $[-B+1,B]$; there are $(2B)^n$ such polynomials.
  Let $P_n(B)$ denote the proportion of polynomials
  in $\mathcal{F}_{n,B}$ which are irreducible and  have Galois group $S_n.$ Then with $c = \inv{6(n-2)}$, we have 
 $1- P_n(B) \ll  B^{-c/\log\log B}$ as $B \to \infty$.
 \end{theorem}

In 1973  Gallagher \cite{Gallagher:1973} applied the large sieve to obtain a
 better bound, showing that
  $ P_n(B) = 1-  O \left( \frac{\log B}{\sqrt{B}} \right)$
 as $B \to \infty.$  Recently  Dietmann \cite{Dietmann:2011} obtained a further improvement 
 on the remainder term, to $ P_n(B) = 1-  O_{\epsilon}\left(B^{-(2-\sqrt{2}) + \epsilon} \right)$.
 Zywina \cite[Theorem 1.6]{Zywina:2010} obtained an improvement of the error term for Galois groups
 not equal to $S_n$ or $A_n$.
 In the direction of number fields, in  1979 Cohen \cite{Cohen:1979} extended this result 
 to number fields in several directions, the simplest of which
 considers polynomials with coefficients in the ring of integers of a number field $k$,
 and considered $S_n$-extensions of $k$ 
  that are obtained by adjoining  the roots of such polynomials. See Cohen \cite{Cohen:1981} for further
 generalizations of this result.
 
  We make the
  following  comparison between random polynomials in $\sO_\fp[x]$ and random polynomials in $\Z[x]$, summarized
in the following Table ~\ref{t:comp}. Its entries on the $\fp$-adic side are
 based on Theorems \ref{t:uncond} -- \ref{c:uncondcor}.

\begin{table}[!hbtp]
\begin{tabular}{|p{2.3 in}|p{2.3 in}|}
\hline
Integer Polynomials & $\fp$-Adic Polynomials\\ \hline
Degree $n$, monic polynomials with integer coefficients bounded by $B$ & Degree $n$, monic polynomials with $\fp$-adic integer coefficients \\ \hline
Probability the polynomial is irreducible approaches 1 & Probability the polynomial is irreducible approaches $\inv{n}$ \\ \hline
Limiting distribution of Galois groups is probability 1 of being $S_n$ & Limiting distribution of Galois groups is probability 1 of being a cyclic subgroup of $S_n$ \\ \hline
The limit of distributions is taken as $B \rightarrow \infty$ & The limit of distributions is taken as $q \rightarrow \infty$\\ \hline
\end{tabular}
\caption{Summary of random integer polynomial results and random $\fp$-adic polynomial results.\label{t:comp}}
\end{table}

   In a different quantitative
    direction,  considering number fields instead of polynomials, Bhargava \cite{Bhargava:2007} recently
    studied the question of  counting  the number of degree $n$ number fields having field discriminant
    at most $X$ and having Galois closure with Galois group $S_n$. He advanced conjectures for their
    number being asymptotic to $c_n X$ as $ X \to \infty$, for specific constants $c_n$, and for the fraction
    of these fields having a given splitting behavior above a fixed prime $p$.
        He showed these conjectures  hold for $n \le 5$, using results of Davenport and Heilbronn \cite{Davenport:1971}
    for the case $n=3$, and
    using his own work   (\cite{Bhargava:2005}, \cite{Bhargava:2010}) on discriminants of quartic and quintic rings
       for the cases $n=4,5$.
          Bhargava's  work is  partially based on  $p$-adic mass
    formulas for \'{e}tale extensions of local fields. 
         His work  can also be viewed as  investigating splitting properties
    of primes in extensions of local fields.

     In a  separate paper with J.~C.~Lagarias we consider  the  splitting behavior above
 a fixed prime $(p)$ of    monic $S_n$-polynomials over $\Z[X]$,
 randomly drawn as in the van der Waerden model above.  We determine the 
 limiting probability distributions as $B \to \infty$ for all $n$ and $p$.
 Rather interestingly the  resulting distributions do not match 
  those predicted in the 
conjectures of Bhargava \cite{Bhargava:2007} above for the fraction of $S_n$-extensions of $\Q$ of
discriminant below a bound $X$ having
prescribed splitting behavior over a fixed prime $p$. In particular, the probabilities do not agree 
in some cases where Bhargava's conjectures are proved; this shows the random number field model
and random polynomial models are truly different.

\subsection{Contents of paper}
In section \ref{s:haar} we prove preliminary lemmas.
Then we prove Theorems \ref{t:uncond}, \ref{t:uncond2}, \ref{c:uncond} and \ref{c:uncondcor}
 in section \ref{s:gal}. Finally in section \ref{s:examp} 
  we calculate the probabilities $\monek(\pgpk(G))$ for $n \in \{2,3\}$.
 \\

%
%
%
\section{Preliminary Results}\label{s:haar}

In this section we prove some preliminary results which will be used in the main proofs in sections  \ref{s:gal} and \ref{s:examp}.

We first remark on the measurability of all sets for which we are computing probabilities.
These sets are subsets of coefficients $(a_0, a_1, ..., a_{n-1} ) \in (\sO_{\fp})^n$. The
condition for a monic 
polynomial $f(x)$ to have a repeated root is that the coefficients belong to the  {\em discriminant locus} 
\[
V_{n,\fp} := \{ (a_0, a_1, ..., a_{n-1})\in (\sO_{\fp})^n : Disc(f) =0\},
\]
where  $Disc(f)$ is a certain
 multivariate polynomial in the $a_i$'s which has integer coefficients.
The set $V_{n,\fp}$ is a  closed subset of $(\sO_\fp)^n$
having measure zero in the product  measure; thus its complement
$(\sO_\fp)^n \smallsetminus V_{n, \fp}$ is a full measure open subset of $(\sO_\fp)^n$.
We now observe that all sets that we consider have the property that 
their intersection  with  $(\sO_\fp)^n \smallsetminus V_{n, p}$ is an
 open set in $(\sO_\fp)^n$.  This is  a consequence of Krasner's lemma
(\cite[Prop. 5.4]{Narkiewicz:1990}), which
states that if $a, b \in \overline{K}_{\fp}$ are such that 
$\abs{a-b}<\abs{a - \sigma(a)}$ for any element 
$\sigma \in \Gal(\overline{K}_{\fp}/K_{\fp})$ 
with $\sigma(a) \neq a$, then $a \in K_{\fp}(b).$
Now, if the roots of  our monic polynomial $f(x) \in K_{\fp}(x)$
are all distinct, then since the roots of a $\fp$-adic polynomial are continuous functions of the coefficients $a_i$
(viewed in an appropriate extension field of $K_{\fp}$),  a
sufficiently small perturbation of the coefficients will induce a small
enough perturbation of the roots for Krasner's lemma to guarantee 
no change in the splitting type of the polynomial and of  the splitting
field generated by the roots.   This verifies the open set property, which in
turn implies that all the sets we consider are measurable for the 
product measure on $(\sO_{\fp})^n$.

The following lemmas present some properties of $\fp$-adic random variables which are distributed according to the Haar measure. We show that the modulo $\fp$ reduction of a random polynomial in $\sO_{\fp}[x]$ is a random polynomial in $\F_q[x]$.

\begin{lemma}\label{l:haar}
A  random variable $X$ on $\sO_{\fp}$ is distributed according to the normalized Haar measure 
if and only if its modulo $\fp^n$ reduction is uniform for all integers $n>0$.
\end{lemma}

\begin{proof}
The Haar measure $\haark$, normalized so that $\haark(\sO_{\fp}) = 1$, is the unique probability measure on $\sO_{\fp}$
 which is additively invariant \cite[Theorem 1-8]{Ramakrishnan:1999}. Thus for any $n>0$ and any $a,b \in \sO_{\fp}$ 
 it holds that 
 \[
 \haark(a + \fp^n\sO_{\fp}) = \haark(b + \fp^n\sO_{\fp}).
 \] 
 This means that for all $n>0$, a $\fp$-adic random variable $X$ 
 which is distributed according to the Haar measure is uniform modulo $\fp^n$.

Conversely, if $X$ is a $\fp$-adic random variable which is uniform modulo $\fp^n$ for any $n > 0$, then the probability 
$X$ is in $a + \fp^n\sO_\fp$ is equal to the probability it lies in $b + \fp^n\sO_\fp$ for any $a, b \in \sO_\fp$. 
A basis of the topology of $\sO_\fp$ are the open sets $a + \fp^n\sO_\fp$ for any $a \in \sO_{\fp}$ and any $n > 0$. 
Since $X$ is additively invariant with respect to these, it is additively invariant. 
Thus by uniqueness, it must be distributed with respect to the Haar measure.
\end{proof}

\begin{lemma}\label{l:indep}
If $X$ and $Y$ are independent random variables on $\sO_{\fp}$, and $X$ is distributed according to the 
normalized Haar measure of $\sO_{\fp}$, then $X+Y$ is as well.
\end{lemma}

\begin{proof}
By the previous lemma, we need only show that $X+Y$ is uniformly distributed modulo $\fp^n$ for all $n \ge 1$. Since we
have shown in the previous lemma that $X$ is uniformly distributed modulo $\fp^n$, it must be that $X+Y$ is as well.
\end{proof}

We will also need the following lemma to facilitate the calculation of $\moneTwok(\pgpTwok(G))$ and $\moneThreek(\pgpThreek(G))$ done in section \ref{s:examp}.


\begin{lemma}\label{l:transf}
For a fixed integer $n>1$ any prime $p\nmid n$ and any $p$-adic field $K_\fp$, let $f(x) = x^n + a_{n-1}x^{n-1}+\cdots+a_0$ with the $a_i$ independent $\fp$-adic random variables distributed according to the Haar measure. Then the polynomial 
\[
f\left(x - \frac{a_{n-1}}{n}\right) = x^n + A_{n-2}x^{n-2} + A_{n-3}x^{n-3} + \cdots + A_0
\]
 is such that $(A_0, \ldots, A_{n-2})\in (\sO_\fp)^{n-1}$ is distributed with respect to the Haar measure on $(\sO_\fp)^{n-1}. $

Additionally for any group $G$, the probability that $G$ is isomorphic to the Galois group of 
$ x^n + A_{n-2}x^{n-2} + A_{n-3}x^{n-3} + \cdots + A_0$ is equal to $\monek(\pgpk(G))$. Likewise, the probability that $\ssy \in \tns$ is the splitting type of $x^n + A_{n-2}x^{n-2} + A_{n-3}x^{n-3} + \cdots + A_0$ is equal to $\monek(\mpik).$
\end{lemma}

\begin{proof}
For any $0\le i \le n-2$ 
\[
A_i = a_i + \rho_i(a_{i+1},\ldots,a_{n-1}) = a_i + \psi_i\left(A_{i+1}, \ldots, A_{n-2}, a_{n-1}\right)
\] 
where $\rho_i$ and $\psi_i$ are polynomials with $\sO_\fp$-coefficients
depending only on $n$ and $i$. Note that if $p | n$, then there may be an $i$ with $A_i$ not in $\sO_\fp$. 
Since $a_i$ is independent of $\rho_i(a_{i+1},\ldots,a_{n-1})$ for every $0\le i \le n-2$ 
and $a_i$ is distributed according to the Haar measure, then by Lemma \ref{l:indep} the 
$A_i$ are identically distributed $\fp$-adic random variables with respect to the Haar measure.

To see that $A_i$ is independent of $A_{i+1}, \ldots, A_{n-2}$ for every $0 \le i \le n-3$ notice that given any 
$b_i, \cdots, b_{n-2} \in \sO_\fp$ and any $n\ge 2$
 \begin{align*}
 \text{Prob}\left[A_i \equiv b_i \pmod {(\fp)^n}\mid A_{i+j} \equiv b_{i+j} \pmod {(\fp)^n} \mbox{ for all } 1\le j\le n+2- i\right] =&
 \\ =\text{Prob}\left[a_i \equiv b_i - \psi_i(b_{i+1}, \cdots, b_{n-2}, a_{n-1}) \pmod {(\fp)^n}\right] 
 =& \inv{q^n}.
 \end{align*}

Finally, note that the Galois groups and splitting types of $f(x)$ and $f\left(x - \frac{a_{n-1}}{n}\right)$ are the same as the two polynomials are linear transformations of each other. Thus the induced distributions on Galois groups and splitting types are the same.
\end{proof}

The following lemma will be used to bound products which appear in several of the later proofs.


\begin{lemma}\label{l:boundprod}
Fix $n$ a positive integer. Given a real number $r > n$ and real numbers $x_1,\ldots, x_n$ with $\max_{1\le i\le n}\abs{x_i} \le \inv{r}$ then \[\abs{\prod_{i=1}^n(1+x_i) - 1} < \frac{n}{r-n}.\]
\end{lemma}

\begin{proof} 
We have that 
\[
\left(1 - \inv{r}\right)^n \le \prod_{i=1}^n(1+x_i) \le \left(1 + \inv{r}\right)^n.
\] 
Note that since $r \ge 1$ none of the terms are negative. So it suffices to prove that
 \begin{equation}\label{e:big}
 \left(1 + \inv{r}\right)^n - 1 \le \frac{n}{r-n}\end{equation} 
 and 
 \begin{equation}\label{e:small}
 1 - \left(1 - \inv{r}\right)^n \le \frac{n}{r-n}.
 \end{equation} 
 Since $r>n$, the we can express $\frac{n}{r-n}$ as the sum of a geometric series: 
 \[\frac{n}{r-n} = \frac{\frac{n}{r}}{1-\frac{n}{r}}= \frac{n}{r} + \frac{n^2}{r^2} + \frac{n^3}{r^3} + \cdots\] 
 The left side of equation (\ref{e:big}) can be expanded as 
 \[
 \left(1 + \inv{r}\right)^n - 1 = {n\choose{1}}\inv{r} + {n\choose{2}}\inv{r^2} + \cdots +{n\choose{n}}\inv{r^n}.
 \]
Since each ${n\choose{i}} < n^i$ we have that 
\begin{align*}
{n\choose{1}}\inv{r} + {n\choose{2}}\inv{r^2} + \cdots +{n\choose{n}}\inv{r^n} \le&  \frac{n}{r} + \frac{n^2}{r^2}  + \cdots+ \frac{n^n}{r^n}\\
<& \frac{n}{r-n}.
\end{align*} 
Likewise for equation (\ref{e:small}) we can expand the left side as  
\begin{align*}
1 - \left(1 - \inv{r}\right)^n =& {n\choose 1}\inv{r} - {n\choose 2}\inv{r^2} + \cdots - (-1)^n{n\choose n}\inv{r^n}\\
\le& {n\choose{1}}\inv{r} + {n\choose{2}}\inv{r^2} + \cdots +{n\choose{n}}\inv{r^n}\\
<& \frac{n}{r-n}. 
\end{align*}
\end{proof}

%
%
%

\section{Distribution of Galois Groups}\label{s:gal}

In this section we analyze the distribution of splitting types and of Galois groups of random, monic, degree $n$ polynomials 
over a general finite field
$\F_q[x],$ with $q= p^f$. We then use the results to prove Theorems \ref{t:uncond} to \ref{c:uncondcor}
concerning splitting types and Galois groups for $\fp$-adic polynomials.

%

\begin{lemma}\label{l:sqfree}
Fix an integer $n>1$. Let $q=p^f$ where  $p$ is  any prime and $f \ge 1$. 
 Then the probability that a random, monic, degree $n$ polynomial $f(x) \in \F_q[x]$ is square free equals
$1 - \inv{q}.$
\end{lemma}

\begin{proof}
(We are indebted to Mike Zieve for the following proof.)
We will define a surjective $q$-to-$1$ map $\psi$ from the set of monic,
degree $n$ polynomials in $\F_q[x]$ which are not square-free to the set of monic,
degree $n-2$ polynomials in $\F_q[x]$. For any $f(x)\in \F_q[x]$ which is monic,
 degree $n$, and has a repeated root, we can factor it uniquely as
 $f(x) = g(x)(h(x))^2$ where $g(x), h(x) \in \F_q[x]$ are monic, $g(x)$ is square-free, and $\text{deg}(h) > 0$.
 Define \[\psi(f)(x) = g(x)\left(\frac{h(x)-h(0)}{x}\right)^2.\] This is a polynomial,
 because $h(x) - h(0)$ is a non-zero polynomial with constant term $0$, so it is divisible by $x$.
 Additionally, $\psi(f)(x)$ has degree $n-2$ because it is the quotient of a
 degree $n$ polynomial and $x^2$.

Given an $r(x) \in \F_q[x]$ which is monic and degree $n-2$, there is a unique factorization of it as $r(x) = g(x)\left(h(x)\right)^2$ where $g(x), h(x) \in \F_q[x]$ are monic (and if $h(x)$ is constant then $h(x) = 1$) and $g(x)$ is square-free. Then $\psi(f)(x) = r(x)$ if and only if $f(x) = g(x)\left(xh(x) + c\right)^2$ for some $c\in \F_q$.

Since we have that $\psi$ is $q$-to-$1$, and there are $q^{n-2}$ monic degree $n-2$ polynomials in $\F_q[x]$, there must be $q*q^{n-2} = q^{n-1}$ monic degree $n$ polynomials with repeated roots in $\F_q[x]$. Thus the proportion of these to all monic degree $n$ polynomials in $\F_q[x]$ is \[\frac{q^{n-1}}{q^n} = \inv{q}.\] So the probability a random, monic, degree $n$ polynomial in $\F_q[x]$ does not have repeated roots (and hence is square-free) is \[1 - \inv{q}.
\]
\end{proof}

The following result is well known \cite[p. 28]{Stanley:1997}:

\begin{lemma}\label{d:cycprob}
Given a fixed positive integer $n,$ and splitting type $\ssy \in \tnss$, the
probability of a uniformly drawn element of $S_n$ falling in conjugacy class $C_{\ssy}$
is
\[\nssy := \frac{C_{\ssy}}{n!} = \prod_{i=1}^n\frac{i^{-c_i(\ssy)}}{c_i(\ssy)!}.\]
\end{lemma}

We now compute the probability that a random, monic, degree $n$ polynomial in $\F_q[x]$ is square-free and has prescribed splitting type.  Recall that for any prime power $q$ there are
\[
\niq =\inv{i}\sum_{d|i}\mu(i/d)q^d
\]
monic, irreducible polynomials in $\F_q[x]$ of any degree $i \ge 1$ \cite[p. 13]{Rosen:2002}.
In what follows we
 adopt the convention that ${A\choose 0} = 1$ for any non-negative integer $A$, and ${A\choose B} = 0$ whenever $B > A$.


\begin{lemma}\label{l:factortype}
Fix a positive integer $n \ge 2$, a prime power $q=p^f$, and a splitting type $\ssy\in \tnss$. Let  $Q_{n,q}(\ssy)$ be the number of monic, square-free, degree $n$ polynomials in $\F_q[x]$ which have
 $\ssy$ as their splitting type. Then
  \[
  Q_{n,q}(\ssy) = \prod_{i=1}^n{\niq \choose c_i(\ssy)}.
  \]
Additionally,
if $n \ge 3$ and  $q > n^2$, then
\[
\abs{\inv{q^{n}}{Q}_{n,q}(\ssy) - \nssy} <  \frac{n^2}{q-n^2}\nssy.\]
If $n = 2$ and $q > n^2 + n$, then
\[\abs{\inv{q^{n}}{Q}_{n,q}(\ssy) - \nssy} < \frac{n^2+n}{q-(n^2 + n)} \nssy.\]

\end{lemma}

\begin{proof}
Since $Q_{n,q}(\ssy)$ counts the number of ways of choosing combinations of distinct irreducible polynomials of prescribed degrees, we have that 
\begin{equation}\label{e:produ}Q_{n,q}(\ssy)=\prod_{i=1}^n{\niq\choose c_i(\ssy)}.\end{equation}

Note that
\[\frac{\niq}{q^i} = \inv{i}(1 + \varepsilon(i,q))
\] with
\[\abs{\varepsilon(i,q)} \le \begin{cases} 0 \mbox{ if } i = 1\\ \inv{q} + \inv{q^2} + \cdots \inv{q^i} < \inv{q-1} \mbox{ if }i \ge 2.\end{cases}
\]
This is because for $i \ge 2$
 \begin{align*}
 \abs{\varepsilon(i,q)} =& \bigg|\inv{q^i}\sum_{\substack{d<i\\ d|i}}\mu(m/d)q^d\bigg|\\
\le& \inv{q^i}\sum_{d=1}^{i-1} q^d = \inv{q} + \inv{q^2} + \cdots + \inv{q^i}.
\end{align*}

With this, we can rewrite equation (\ref{e:produ}) as
\[\frac{Q_{n,q}(\ssy)}{q^n}
= \prod_{i=1}^n\frac{1}{c_i(\ssy)!}\prod_{\ell=0}^{c_i(\ssy)-1}\frac{(\niq - \ell)}{q^i}
= \prod_{i=1}^n\frac{i^{-c_i(\ssy)}}{c_i(\ssy)!}\prod_{\ell=0}^{c_i(\ssy)-1}\left(1 + \varepsilon(i,q) - \frac{i\ell}{q^i}\right).\]

Note that since the sum $\sum_{i=1}^n ic_i(\ssy) = n$ we have for all $\ell$ with
$0 \le \ell \le c_i(\ssy) - 1$ that $i\ell < n$.
Also note that for $0\le \ell \le c_i(\ssy) -1$ we can write
 \begin{equation}\label{e:combi}\frac{\niq-\ell}{q^i} =
 \inv{i}(1 + \tilde{\varepsilon}_\ell(i,q))
 \end{equation}
with, for $i \ge 2$ and $q\ge n \ge 2$
\[\abs{\tilde{\varepsilon}_\ell(i,q)} \le \abs{\varepsilon(i,q)} + \frac{i\ell}{q^i} < \inv{q-1} + \frac{n}{q^2} \le \inv{q-1} + \inv{q} \le \frac{2}{q} + \frac{1}{q(q-1)}, \] and if $i = 1$ then \[\abs{\tilde{\varepsilon}_\ell(i,q)} = \frac{\ell}{q} \le \frac{n}{q}.\]
Thus by Lemma \ref{d:cycprob}
\begin{equation}
\abs{\inv{q^{n}}Q_{n,q}(\ssy) - \nssy} = \nssy\abs{\prod_{i=1}^{n}\prod_{\ell=0}^{c_i(\ssy)-1}(1 +  \tilde{\varepsilon}_\ell(i,q))-1}.
\end{equation}
When $n \ge 3$ all the terms $\tilde{\varepsilon}_\ell(i,q)$
have absolute value less than or equal to $\frac{n}{q}$.
When $n = 2$ all the $\tilde{\varepsilon}_\ell(i,q)$ have absolute value less than or equal to $\frac{n+1}{q}$. Thus when $n \ge 3$ and $q \ge n^2$ we can apply
 Lemma \ref{l:boundprod} with $r = \frac{q}{n}\ge n$
 and conclude that
 \[
 \abs{\inv{q^{n}}Q_{n,q}(\ssy) - \nssy } < \nssy \frac{n^2}{q - n^2}.
 \]
  When $n = 2$ and $q \ge n^2 + n$ we can apply
  Lemma \ref{l:boundprod} with $r = \frac{q}{n + 1} \ge n$
   and conclude that
   \[
   \abs{\inv{q^{n}} Q_{n,q}(\ssy) - \nssy} < \nssy \frac{n^2 + n}{q - (n^2 + n)}.
   \]
\end{proof}

We now have the necessary tools to prove Theorem \ref{t:uncond} through
Theorem  \ref{c:uncondcor}.


\begin{proof}[of Theorem~\ref{t:uncond}]
By Lemma \ref{l:haar} we have over $K_{\fp}$ that the modulo $\fp$ reduction of a degree $n$, monic, 
$\fp$-adic, random polynomial is equidistributed among the 
monic, degree $n$ polynomials in $\F_q[x]$, with $q= p^{f_{\fp}}$.
 Thus $\monek(\pgpsk)$ is equal to the probability a random monic, degree $n$ polynomial in 
 $\F_q[x]$ is square-free. By Lemma \ref{l:sqfree} this probability is $1 - \inv{q}$, thus
\[\monek(\pgpsk) = 1-\inv{q}.\]

Hensel's lemma \cite[p. 129]{Neukirch:1999} implies that if $f(x) \in \sO_{\fp}[x]$ is monic and has modulo $\fp$ reduction 
$\overline{f}(x)$ which is square-free, then the splitting types
$\ssy(\overline{f})$ and $\ssy(f)$ are equal.
Given any splitting type $\ssy \in \tnss$, Lemma \ref{l:factortype} shows that
for any prime power $q=p^f$, the probability that $\overline{f}$ is square-free and
has splitting type $\ssy$  equal to
$\inv{q^n}\prod_{i=1}^n{\niq \choose c_i(\ssy)}.$
This is $\monek(\pgpsk(\ssy))$. \end{proof}


\begin{proof}[of Theorem~\ref{t:uncond2}]
  Lemma \ref{l:factortype} shows that
 \begin{equation}\label{e:thm1.2}
\abs{\inv{q^{n}}
 \prod_{i=1}^n{\niq \choose c_i(\ssy)} - \nssy} <
 \begin{cases}\displaystyle\frac{n^2+ n}{q -(n^2 + n)}\,\nssy ~~\mbox{ if } n = 2 \mbox{ and }
 pq\ge n^2 + n, \\
 \displaystyle\frac{n^2 }{q- n^2} \, \nssy ~~\mbox{ if } n \ge 3 \mbox{ and } q \ge n^2.\end{cases}
  \end{equation}
  Substituting in $\monek(\pgpsk(\ssy)) = \inv{q^{n}}\prod_{i=1}^n{\niq \choose c_i(\ssy)}$
  from Theorem \ref{t:uncond} completes the proof.
\end{proof}


\begin{proof}[of Theorem~\ref{t:uncond2cor}]
As $q=p^f$ grows without bound both terms on the right side of equation (\ref{e:thm1.2}) goes to zero on the order of 
$\inv{q}$ and so $\monek(\mpistk) \rightarrow \nssy$ as $q\rightarrow \infty$.
\end{proof}


\begin{proof}[of Theorem~\ref{c:uncond}]
Let $f(x)$ be a monic, degree $n$ polynomial in $\sO_{\fp}[x]$ with square-free reduction
modulo $\fp$. Then the Galois group of the splitting field of $f(x)$ over $K_{\fp}$ is cyclic. Furthermore, if $\ssy$
is the splitting type of $f(x)$, its Galois group is generated by an element of
$C_\ssy \subset S_n$. Therefore, the Galois group of $f(x)$ is conjugate to a given
cyclic subgroup $G \subset S_n$ if and only if the splitting type of $f(x)$ is equal to $\ssy_G$.

So for a cyclic group $G \subset S_n$, \[\monek(\pgpsk(G)) = \monek(\pgpsk(\ssy_G)).\]

If $n \ge 3$, and $q > n^2$, we have by an application of Theorem \ref{t:uncond} and the definition of $\tv_n(G)$ that
\begin{align}\label{a:uncond3}
\abs{\monek(\pgpsk(G)) - \tv_n(G)} &=
 \abs{\monek(\pgpsk(\ssy_G)) - \nu_n(\ssy_G)}\\ &< \frac{n^2}{q -  n^2}\nu_n(\ssy_G)
 = \frac{n^2}{q -  n^2} \tv_n(G).\end{align}

If $n = 2$ and $q > n^2 + n$ we instead have that
\begin{equation}\label{e:uncond2}
\abs{\monek(\pgpsk(G)) - \tv_n(G)} < \frac{n^2 +  n}{ q -( n^2 + n)} \tv_n(G).\end{equation}

We also showed in Theorem \ref{t:uncond}
that the measure of all degree $n$, monic polynomials in $\sO_{\fp}[x]$
 with discriminant divisible by $\fp$ is equal to $\inv{q}$. Thus the probability that a polynomial has cyclic Galois group and has 
 discriminant divisible by $\fp$ is at most $\inv{q}$.
 Combining this with equations (\ref{a:uncond3}), and (\ref{e:uncond2})
  by the triangle inequality we have that
\[
\abs{\monek(\pgpk(G)) - \tv_n(G)} <
\begin{cases} \displaystyle
\frac{ n^2 + n}{q-( n^2 + n)}\, \tv_{n}(G) + \inv{q} \mbox{ if } n = 2 \mbox{ and }q > n^2 + n,\\ \displaystyle\frac{n^2}{q- n^2}\,\tv_{n}(G) + \inv{q} \mbox{ if } n \ge 3 \mbox{ and } q > n^2. \end{cases}\]

Finally, the only way a monic polynomial $f(x) \in \sO_{\fp}[x]$ can have non-cyclic
Galois group is if its modulo $\fp$ reduction is not square-free.
By Lemma \ref{l:sqfree} the probability of this occurring is at most
$\inv{q}.$ Thus $\sum_{G \mbox{ non-cyclic}}\monek(\pgpk(G)) \le \inv{q}.$
\end{proof}


\begin{proof}[of Theorem~\ref{c:uncondcor}]
This is immediate from the bounds obtained  in  Theorem \ref{t:uncond}.
\end{proof}
%
%
%

\section{Low Degree Cases}\label{s:examp}

We obtain exact formulas for degrees $n=2$ and $n=3$ for general $p$-adic fields  $K_{\fp}$, when $p>n$.
We compute $\monek(P_{n,\fp}(\cdot))$ for all splitting types
and all Galois groups, noting that  exact  formulas for $\monek(P_{n,\fp}^{\ast}(\cdot))$ were given
in Theorem ~\ref{t:uncond}. The restriction that $p> n$ guarantees that  no splitting field of a 
degree $n$ polynomial
can be wildly ramified. The case $n=2$ was previously analyzed for $\Q_p$ in
the 1999 Ph.D. thesis of Limmer \cite[Corollary, p. 27]{Limmer:1999}. The proof given here is based on
similar ideas but differs in some details.

%
%
%

\subsection{Degree $n=2$} \label{sec41}

\begin{theorem}\label{t:examp2}
Let $p > 2$ be any prime, and let $K_{\fp}$ be any finite extension of $\Q_p$ with residue field $\F_q$. Then 
for monic quadratic polynomials with $\sO_{\fp}$ coefficients we have 
\begin{align*}
\moneTwok(P_{2,\fp}(\text{id})) =\moneTwok(P_{2,\fp}(1,\ 1))=& \inv{2} - \inv{2q+2}\\
\moneTwok(P_{2,\fp}(\Z/2\Z)) =\moneTwok(P_{2,\fp}(2))=& \inv{2} + \inv{2q+2}. 
\end{align*}
\end{theorem}

\begin{proof}

By Lemma \ref{l:transf} we can reduce to studying random polynomials of the form $f(x) = x^2 + A \in \sO_\fp[x]$.
 If $2r$ is the largest even power of $\pi_\fp$ which divides $A$, which is a $q^{-2r}-q^{-2r-2}$ event, 
 then $g(x) :=\pi_\fp^{-2r}{f\left(\pi_\fp^rx\right)}$ is a polynomial in $\sO_\fp[x]$. 
 Furthermore, $g(x)$ and $f(x)$ have the same splitting type and Galois group.
  The coefficients of $g(x) = x^2 + B$ are distributed with respect to the Haar measure 
  conditional on the fact that the constant term is not divisible by $\pi_\fp^2$.

 If we let $Y$ denote the conditional probability that $x^2 + B \in \sO_\fp[x]$ factors into two linear polynomials given that 
 $B \notin \pi_\fp^2\sO_p$, then we have the following equality:\[\moneTwok(P_{2,\fp}(1,\ 1))=\moneTwok(P_{2,\fp}(\text{id})) = Y\sum_{i=0}^\infty\left(\inv{q^{2i}} - \inv{q^{2i+2}}\right) = Y.\]

We then have to study the distribution on Galois groups induced by the random polynomial $x^2 + B\in\sO_\fp[x]$, when 
$B\notin \pi_\fp^2\sO_{\fp}$. Since we are studying the conditional probability, we will only be calculating measure of subsets of 
$\sO_\fp \smallsetminus \pi_\fp^2\sO_\fp$, and so will need to divide our measures by $1 - q^{-2}$ to obtain probabilities. 
If $B \notin \pi_\fp\sO_\fp$, which is a probability $1 - \inv{q}$ event, then the Galois group is trivial (and has splitting type $(1,\ 1)$)
 if $B$ is a quadratic residue modulo $(\fp)$, and it is $\Z/2\Z$ (with splitting type $(2)$) otherwise; both of these are equally likely events. If $B \in \pi_\fp\sO_\fp \smallsetminus \pi_\fp^2\sO_\fp$, 
 which is a probability $\inv{q} - \inv{q^2}$ event, then the extension $K_f$ must be ramified hence the Galois group 
 is $\Z/2\Z$. Thus we have that
  \[Y = \frac{\inv{2}\left(1 - \inv{q}\right) + 0\left(\inv{q} - \inv{q^2}\right)}{1 - \inv{q^2}} = 
  \inv{2}\left(\inv{1 + \inv{q}}\right) = \inv{2} - \inv{2(q+1)}.
  \]
\end{proof}

We also wish to analyze the case where the splitting field of a polynomial is unramified over $K_\fp$.
The splitting field can sometimes be unramified,
even when $\pi_\fp$ divides the discriminant of the polynomial.
Let $R_{2,\fp}$ be the set of coefficients  $(a_2, a_1, a_0)$ of monic quadratic polynomials in $\sO_\fp[x]$
whose splitting fields are unramified over $K_\fp$.
For any $\ssy \in \tnss$, let $R_{2,\fp}(\ssy)$ be the subset of
$R_{2,\fp}$ associated to polynomials with splitting type $\ssy$.

\begin{theorem}\label{t:unramifiedext2}
Fix a prime $p > 2$ and let $K_{\fp}$ be any finite extension of $\Q_p$ with residue field $\F_q$. Then
 \[
 \moneTwok(R_{2,\fp}) = 1- \displaystyle\frac{1}{q+1}.
 \]
Here
\begin{align*}
\moneTwok(R_{2,\fp}(1,\ 1)) =& \inv{2} - \inv{2q+2}\\
\moneTwok(R_{2,\fp}(2)) =&  \inv{2} - \inv{2q+2}.\\
\end{align*}
\end{theorem}

\begin{proof}
As in the previous theorem, we reduce to analyzing polynomials of the form $x^2 +B \in \sO_{\fp}[x]$ with $B \notin \pi_\fp^2\sO_{\fp}$. Whenever $B \notin \pi_\fp\sO_{\fp}$, the modulo $p$ reduction of $x^2 +B$ is a square-free polynomial and the splitting field of
 $x^2 +B$ is therefore unramified over $K_{\fp}$. Furthermore, the probabilities that $x^2 + B$ factors or is irreducible in this case are equal.

The field extension is ramified if and only if  $B$ is exactly divisible by an odd power of $\pi_\fp$.
The measure of the cases where $\pi_\fp^{2k}$ for $k \ge 1$ exactly divide $B$ add up to  $\frac{q^2}{q^2-1}$,  so we conclude that:
\begin{align*}
 &\moneTwok(R_{2,\fp}) = 1- \frac{q^2}{q^2-1}\frac{q-1}{q^2} = 1- \inv{q+1}\\
 &\moneTwok(R_{2,\fp}(1,\ 1)) =
\moneTwok(R_{2,\fp}(2)) =  \inv{2} - \inv{2q+2}.
\end{align*}
\end{proof}

This result shows that the probability $P_{2,\fp}^{unr}$ that the field generated by the roots
is trivial, conditioned on requiring that the field extension be unramified, is, for $p >3$ and any $p$-adic field $K_\fp$,
\[
P_{2, \fp}^{unr}= \frac{ \moneTwok (R_{2, \fp}(1, 1))}{\moneTwok(R_{2, \fp})} =\frac{1}{2}.
\]
This conditional probability is independent of $p$ and $K_\fp$.

%
%
%

\subsection{Degree $n=3$} \label{sec42}

We use the notation $\left(\frac{a}{b}\right)$ to denote the Legendre symbol.

\begin{theorem}\label{t:examp3}
Fix a prime $p > 3$ and let $K_{\fp}$ be any finite extension of $\Q_p$ with residue field $\F_q$.
Then for monic cubic polynomials with $\sO_{\fp}$ coefficients we have
 \begin{align*}
\moneThreek(\pgpThreek(\text{id}))=\moneThreek(\pgpThreek(1,\ 1,\ 1)) =& \inv{6} - \frac{3q^4+q^3+2q^2+2q+1}{6(q+1)(q^4+q^3+q^2+q+1)}\\
\moneThreek(\pgpThreek(\Z/2\Z))=\moneThreek(\pgpThreek(2,\ 1)) =& \inv{2} + \frac{q^4-q^3-1}{2(q+1)(q^4 + q^3 + q^2 +q+1)}\\
\moneThreek(\pgpThreek(3)) =& \inv{3} + \frac{2q^2 -q +2}{3(q^4 + q^3 + q^2 + q + 1)},
\end{align*}
and
\begin{align*}
\moneThreek(\pgpThreek(\Z/3\Z)) =& \inv{3} -\frac{q^2+q+1}{3(q^4 + q^3 + q^2 + q + 1)} + \inv{2}\left(1+\left(\frac{q}{3}\right)\right)\frac{q^2+1}{q^4+q^3+q^2+q+1}\\
\moneThreek(\pgpThreek(S_3)) =&\inv{2}\left(1-\left(\frac{q}{3}\right)\right)\frac{q^2+1}{q^4+q^3+q^2+q+1}.
\end{align*}
\end{theorem}

We also wish to analyze the case where the splitting field of a polynomial is unramified over $K_\fp$. This can occur even 
when $\pi_\fp$ divides the discriminant of the polynomial. Let $R_{3,\fp}$ be the set of coefficients of 
monic cubic polynomials in $\sO_\fp[x]$ whose splitting fields are unramified over $K_\fp$. For any $\ssy \in \tnss$, let $R_{3,\fp}(\ssy)$ be the subset of $R_{3,\fp}$ associated to polynomials with splitting type $\ssy$.


\begin{theorem}\label{t:unramifiedextn3}
Fix a prime $p > 3$ and let $K_{\fp}$ be any finite extension of $\Q_p$ with residue field $\F_q$. Then 
\[\moneThreek(R_{3,\fp}) = 1- \displaystyle\frac{q^2 + q+1}{q^4 + q^3 + q^2 + q + 1}\] and\begin{align*}
\moneThreek(R_{3,\fp}(1,\ 1,\ 1)) =& \inv{6} - \frac{3q^4+q^3+2q^2+2q+1}{6(q+1)(q^4+q^3+q^2+q+1)},\\
\moneThreek(R_{3,\fp}(2,\ 1)) =& \inv{2} + \frac{q^4 -q^3 - 2q^2 -2q -1}{2(q + 1)(q^4 + q^3 + q^2 + q + 1)},\\
\moneThreek(R_{3,\fp}(3))  =& \inv{3} - \frac{q^2 + q + 1}{3(q^4 + q^3 + q^2 + q + 1)}.
\end{align*}
\end{theorem}

Note as well that by definition
\[\nu_{3,\fp}(R_{3,\fp})=\moneThreek(R_{3,\fp}(1,\ 1,\ 1)) +\moneThreek(R_{3,\fp}(2,\ 1))+\moneThreek(R_{3,\fp}(3)).\]

This result shows that the probability $P^{unr}_{3,\fp}$ that the field generated by the roots is trivial, conditioned on requiring that the field extension be unramified, is, for $p > 3$ and any $p$-adic field $K_\fp$,
\[P^{unr}_{3,\fp} = \frac{\moneThreek(R_{3,\fp}(1,\ 1,\ 1))}{\moneThreek(R_{3,\fp})} = \frac{q^2 - q + 1}{6(q^2 + 2q + 1)} = \inv{6} - \frac{q}{2(q^2 + 2q + 1)}.\]

In order to prove these results, we will need the following lemma.


\begin{lemma}\label{l:limit}
Fix a prime $p>3$, let $K_{\fp}$ be any finite extension of $\Q_p$ with residue field $\F_q$, and fix
 $a \in \left(\sO_\fp/(\fp)\sO_\fp\right)^*$.
Then for any $n \ge 1$ there are $q^{2n-2}$ polynomials $g(x) = x^3 + Ax + B \in \sO_\fp/(\fp)^n\sO_\fp[x]$ which  satisfy 
${g}(x) \equiv (x+a)^2(x-2a)\mod (\fp)$. If $n$ is even then
\[\frac{q^{n-1}}{2}\left(\frac{q^{n-1}-q}{q+1}\right) + q^{n-1}\] of these polynomials $g(x)$
are expressible as the product of three linear factors, and if $n$ is odd then \[\frac{q^{n-1}}{2}\left(\frac{q^{n-1}-1}{q+1}\right) + q^{n-1}\] of these  polynomials are expressible as the product of three linear factors.

Let $A,B \in \sO_\fp$ such that $x^3 + Ax + B \equiv (x+a)^2(x-2a)~~(\bmod \, (\fp))$. The ratio of the measure in $(\sO_\fp)^2$ of the subset of such polynomials that factor into three linear factors in $\sO_\fp[x]$ to the measure of all such 
$(A,B)$ is exactly $\inv{2q + 2}.$
\end{lemma}

\begin{proof}
There are $q^{2n-2}$ factorizations $(x- \alpha)(x-\beta)(x-\gamma) \in \sO_\fp/(\fp)^n\sO_\fp[x]$ with 
$\alpha + \beta = -\gamma$ and $\alpha \equiv \beta\equiv a \mod (\fp)$. However $\sO_\fp/(\fp)^n\sO_\fp[x]$ is not a unique factorization domain. We call two different triples $(\alpha,\beta,\gamma), (\alpha',\beta',\gamma')\in \sO_\fp/(\fp)^n\sO_\fp$ equivalent if $(x- \alpha)(x-\beta)(x-\gamma) =(x- \alpha')(x-\beta')(x-\gamma')$. 
To count the total number of distinct polynomials $x^3 + ax + b\in \sO_\fp/(\fp)^n\sO_\fp[x]$ which are expressible as a product of three linear factors, we will count equivalence classes of triples $(\alpha,\beta,\gamma) \in \sO_\fp/(\fp)^n\sO_\fp$ with $(\alpha,\beta,\gamma) \equiv (a,a,-2a) (\bmod (\fp))$, and $\alpha + \beta + \gamma = 0$.

If two such triples $(\alpha,\beta,\gamma)$ and $(\alpha',\beta',\gamma')$ are equivalent, then $\gamma$ must equal 
$\gamma'$. This is because Hensel's lemma guarantees a unique lift of $(x-2a)$ to $\sO_\fp/(\fp)^n\sO_\fp[x]$. We are allowed to apply Hensel's lemma because the derivative of $x^3 + ax + b$ is $(x+a)(3x-3a)$. Plugging in $x = 2a$ we get $9a^2$ which is not zero because $p>3$.

Additionally, for the triples to be equivalent,  the coefficients of the two polynomials must be equal. In particular
 \[(x-\alpha)(x+\alpha + \gamma)(x - \gamma) = (x - \alpha')(x+\alpha' + \gamma)(x-\gamma).\] 
 So the constant terms must agree, which is equivalent to:
\begin{align*}
\alpha&\gamma(\alpha +\gamma) = \alpha'\gamma(\alpha' + \gamma)\\
\Longleftrightarrow&\alpha(\alpha +\gamma) = \alpha'(\alpha' + \gamma)\\
\Longleftrightarrow&(\alpha-\alpha')(\alpha'-\beta) = 0.
\end{align*}
Also the linear terms must be equal:
\[
\alpha^2 + \alpha\gamma + \gamma^2 = (\alpha')^2 + \alpha'\gamma + \gamma^2,
\] which is clearly equivalent to the same condition.

For any integer $r \in [1,n-1]$ and any of the $q^{n-1}$ choices of $\alpha$, there are
 $q^{n-r}-q^{n-r-1}$ different elements $\beta$ so that $\pi_\fp^r$ is the largest power of
 $\pi_\fp$ dividing $\alpha - \beta$. If $r < \frac{n}{2}$ then
 $(\alpha - \alpha')(\beta - \alpha') \equiv 0 \mod (\fp)^n$ if and only if
  $\alpha' \equiv \alpha \mod (\fp)^{n-r}$ or $\alpha' \equiv \beta \mod (\fp)^{n-r}$;
  there are $2q^{r}$ such $\alpha'$. If $r \ge \frac{n}{2}$, then
   $(\alpha - \alpha')(\beta - \alpha') \equiv 0 \mod (\fp)^n$
    if and only if $\alpha' \equiv \alpha \equiv \beta \mod (\fp)^{\ceil{\frac{n}{2}}}$,
    and there are $q^{\fl{\frac{n}{2}}}$ such $\alpha' \in \sO_\fp/(\fp)^n\sO_\fp$.

Thus, for any of the $q^{n-1}q^{\fl{\frac{n}{2}}}$ triples $(\alpha, \beta, \gamma)$
with $\pi_\fp^{\ceil{\frac{n}{2}}}|\alpha-\beta$, there are $q^{\fl{\frac{n}{2}}}$
equivalent triples $(\alpha',\beta', \gamma')$. So there are only $q^{n-1}$
 distinct polynomials with $\pi_\fp^{\ceil{\frac n2}}|\alpha - \beta$.

For any $1 \le r \le \fl{\frac{n}{2}}$ there are $q^{n-1}(q^{n-r}-q^{n-r-1})$
 triples $(\alpha, \beta, \gamma)$ with $\pi_\fp^r | \alpha - \beta$ and
 $\pi_\fp^{r+1} \nmid \alpha - \beta$. Every such triple is equivalent to $q^{r}$
  other triples $(\alpha', \beta', \gamma')$. So there are $q^{n-r-1}(q^{n-r}-q^{n-r-1})$ distinct polynomials with 
  $\pi_\fp^r|\alpha - \beta$ and $\pi_\fp^{r+1} \nmid \alpha - \beta$.

So the total number of equivalence classes of triples $(\alpha, \beta, \gamma)$
 is
 \begin{align*}
 q^{n-1} + \sum_{r = 1}^{\fl{\frac{n}{2}}}q^{n-r-1}\left(q^{n-r}-q^{n-r-1}\right) =&
 q^{n-1} + q^n\left(q^{n-1} + q^{n-2}\right)\frac{1-q^{2\fl{\frac{n}{2}}}}{1-q^{-2}}\\
 =& \begin{cases}
  \displaystyle \frac{q^{n-1}}{2}\left(\frac{q^{n-1}-1}{q+1}\right) + q^{n-1} \mbox{ if }n \mbox{ is odd}\\ \displaystyle\frac{q^{n-1}}{2}\left(\frac{q^{n-1}-q}{q+1}\right) + q^{n-1} \mbox{ if } n \mbox{ is even.}\end{cases}
 \end{align*}

The limit of the probabilities a polynomial $x^3 + Ax + B \in\sO_\fp/(\fp)^n\sO_\fp[x]$ factors into three linear polynomials, 
given that $x^3 + Ax + B \equiv x^3 + ax+b\pmod (\fp)$ is equal to the desired ratio in $(\sO_\fp)^2$. From the previous part of the lemma, when $n$ is odd this limit is 
\[\lim_{n\rightarrow \infty}q^{-2n+2}\left[\left(\frac{q^{n-1}}{2}\right)\left(\frac{q^{n-1} - 1}{q+1}\right) + q^{n-1}\right] = \frac{1}{2q+2}.\] When $n$ is even the limit is 
\[\lim_{n\rightarrow \infty}q^{-2n+2}\left[\frac{q^{n-1}}{2}\left(\frac{q^{n-1}-q}{q+1}\right) + q^{n-1}\right] = \frac{1}{2q+2}.
\]
\end{proof}


\begin{proof}[of Theorem~\ref{t:examp3}]
The proof of this theorem  proceeds much like Theorem \ref{t:examp2}.
We again use Lemma \ref{l:transf} to reduce to the case of $f(x) = x^3 + Ax + B \in \sO_\fp[x]$.
 If $A \in \pi_\fp^2\sO_{\fp}$ and $B \in \pi_\fp^3\sO_{\fp}$ then the polynomial
 ${\pi_\fp^{-3}}{f(\pi_\fp x)}$ is in $\sO_\fp[x]$ and has the same splitting type and  Galois group as $f(x)$.
 Thus if $m$ is the largest integer such that $\pi_\fp^{3m}|B$ and $\pi_\fp^{2m}|A$, which is a
 ${q^{-5m}}-{q^{-5m-5}}$ event, then we can instead study $\pi_\fp^{-3m}{f\left(\pi_\fp^mx\right)}$.
 Thus, like in Theorem \ref{t:examp2}, if $Y_G$ is the conditional probability that
 $f(x)$ has Galois group $G$ given that $\pi_\fp^2\nmid A$ or $\pi_\fp^3\nmid B$, then
 \[
 \moneThreek(\pgpThreek(G)) = Y_G\sum_{i=0}^\infty\left(\inv{q^{5i}}-\inv{q^{5i+5}}\right) = Y_G.
 \]
  The same equality holds if we instead studied $Y_\ssy$, the
  conditional probability that $f(x)$ had splitting type $\ssy$. Note that the measure of the conditional probability space is $1 - q^{-5}$.

We use the Newton polygon of the polynomial \cite[p. 144, Prop. 6.3]{Neukirch:1999} with uniformizing parameter $\pi_\fp$
 to determine the ramification of the extension generated by  $x^3 + Ax + B$. The polygon is the upper convex hull of the points $(0, \text{ord}_{\pi_\fp}(B))$, $(1, \text{ord}_{\pi_\fp}(A))$, and $(3, 0)$. The denominators of the slopes of this polygon correspond to the degree of ramification of the roots of the polynomial, and therefore to the degree of ramification of the splitting field over $K_\fp$.
 We will divide the set of possible pairs $(A,B)$ into cases, and compute the unconditional probability of each case occurring. At the end we will divide these probabilities
 by $1-q^{-5}$ to arrive at the conditional probabilities, $Y_G$
  for each subgroup $G \subset S_3$ and $Y_\ssy$ for each $\ssy \in T_3^*$.\smallskip

{\bf Case 1: $B \in \pi_\fp\sO_\fp$ and $A \notin \pi_\fp\sO_\fp$:} The probability of choosing a
polynomial in this case is $\inv{q}\left(1-\inv{q}\right)$.

In this case the reduction modulo $(\fp)$ is $x(x^2 + a)$ where $a \equiv A \mod (\fp)$.
Then this factors into three distinct linear polynomials with probability $\inv{2}$
and $x^2 + a$ is irreducible with probability $\inv{2}$ depending on
whether $a$ is a quadratic residue or not. Since the reduction is always square-free, the splitting field of the polynomial is unramified and
 the splitting type of the modulo $(\fp)$ reduction is equal to the splitting type of $f(x)$.
 In this case, the Galois group of $f(x)$ is trivial or $\Z/2\Z$ with equal probability.
 Therefore, this case contributes the following distribution:
 \begin{align*}
&\ssy(f) = (1,\ 1,\ 1)~~~\mbox{with probability}~~~\inv{2}\left( \frac{1}{q}\left(1- \frac{1}{q}\right)\right)\\
&\ssy(f) = (2,\ 1)~~~\mbox{with probability}~~~\inv{2}\left( \frac{1}{q}\left(1- \frac{1}{q}\right)\right)\\
&\Gal(K_f/K_\fp) = Id~~~\mbox{with probability}~~~\frac{1}{2} \left( \frac{1}{q}\left(1- \frac{1}{q}\right)\right)\\
&\Gal(K_f/K_\fp) = \Z/2\Z~~~\mbox{with probability}~~~\frac{1}{2} \left( \frac{1}{q}\left(1- \frac{1}{q}\right)\right).\end{align*}

{\bf Case 2: $B \in \pi_\fp^2\sO_\fp$ and $A\in \pi_\fp\sO_\fp-\pi_\fp^2\sO_\fp$:} The probability of choosing a polynomial in this case is $q^{-2}\left(q^{-1}-q^{-2}\right)$.

If $B \in \pi_\fp^2\sO_\fp$ and $A \in \pi_\fp\sO_\fp-\pi_\fp^2\sO_\fp$ then
 there would be ramification of degree two, as the slopes of the Newton polygon would be $-\text{ord}_{\pi_\fp}(B) + 1$ and $-\inv{2}$, where $\text{ord}_{\pi_\fp}(B)$ is the number of powers of $\pi_\fp$ which divide $B$. Thus the extension would always be a degree $2$ extension with Galois group $\Z/2\Z$. This is because two of the roots are in a quadratic ramified extension, and the third is in an unramified extension. Thus the polynomial must factor as an irreducible quadratic and a linear, and so the splitting type must be
  $(2,\ 1)$ and the Galois group must be $\Z/2\Z$.
Therefore, this case contributes the following distribution:
\begin{align*}&\ssy(f) = (2,\ 1)~~~\mbox{with probability}~~~\inv{2}q^{-2}\left(q^{-1}-q^{-2}\right)\\
&\Gal(K_f/K_\fp) = \Z/2\Z~~~\mbox{with probability}~~~\inv{2}q^{-2}\left(q^{-1}-q^{-2}\right).\end{align*}

{\bf Case 3: $B \in \pi_\fp^2\sO_\fp-\pi_\fp^3\sO_\fp$ and $A \in \pi_\fp^2\sO_\fp$.}
The probability of choosing a polynomial in this case is $\left(q^{-2}-q^{-3}\right)q^{-2}$.

In this case, the slope of the Newton polygon is $-\frac{2}{3}$. Thus there are three degrees of ramification, and $f(x)$ is has splitting type $(3)$. If the discriminant of $f(x)$ is a square then, then the Galois group is $\Z/3\Z$, otherwise it is $S_3$. The discriminant is $-4A^3 - 27B^2$, thus $\pi_\fp^4$ divides the discriminant. After factoring that out, we are left with
$\inv{\pi_\fp^4}(-4A^3 - 27B^2)$ which is not in $\pi_\fp\sO_\fp$. To determine if it is a square,
we look at its modulo $(\fp)$ reduction which is congruent to $-27\frac{B^2}{\pi_\fp^4}.$
 Thus it is a square if and only if $-27$ is a square modulo $(\fp)$. If $f_\fp$ is even, then every integer is a square modulo $(\fp)$. Otherwise, $-27$ is a square if and only if $-27$ is a square modulo $p$,
 which is if and only if $p \equiv 1 \pmod 3$. Combining these two facts, $-27$ is a square modulo $(\fp)$ if and only if $q \equiv 1 \pmod 3$.
 Therefore, this case contributes the following distribution:
\begin{align*}&\ssy(f) = (3)~~~\mbox{with probability}~~~\left(q^{-2}-q^{-3}\right)q^{-2},\\
&\Gal(K_f/K_\fp) = \Z/3\Z~~~\mbox{with probability}~~~\left(q^{-2}-q^{-3}\right)q^{-2},
~\mbox{if} ~q \equiv 1 ~(\bmod \, 3),\\
&\Gal(K_f/K_\fp) = S_3~~~\mbox{with probability}~~~\left(q^{-2}-q^{-3}\right)q^{-2},
~\mbox{if} ~q \equiv 2 ~(\bmod \, 3).\end{align*}

{\bf Case 4: $B \in \pi_\fp\sO_\fp-\pi_\fp^2\sO_\fp$ and $A \in \pi_\fp\sO_\fp$.} The probability of
choosing a polynomial in this case is $\left(q^{-1} - q^{-2}\right)q^{-1}$.

In this case, the slope of the Newton polygon is $-\inv{3}$. Like before the splitting type is $(3)$ and the extension 
has Galois group $\Z/3\Z$ or $S_3$
depending again on whether or not $q$ is a square modulo $3$.
Therefore, this case contributes the following distribution:
\begin{align*}&\ssy(f) = (3)~~~\mbox{with probability}~~~\left(q^{-1}-q^{-2}\right)q^{-1}\\
&\Gal(K_f/K_\fp) = \Z/3\Z~~~\mbox{with probability}~~~\left(q^{-1}-q^{-2}\right)q^{-1},
~\mbox{if} ~q \equiv 1 ~(\bmod \, 3),\\
&\Gal(K_f/K_\fp) = S_3~~~\mbox{with probability}~~~\left(q^{-1}-q^{-2}\right)q^{-1},
~\mbox{if} ~q \equiv 2 ~(\bmod \, 3).\end{align*}

{\bf Case 5: $B\notin \pi_\fp\sO_\fp$.} The probability of choosing a polynomial in this case is $\left(1- q^{-1}\right)$.

In this case there is no ramification, and so the splitting field is unramified over $K_\fp$ and the Galois group
 is completely determined by the splitting type of the polynomial.
 There are $q(q-1)$ different possible reductions modulo $(\fp)$,
 and we consider the probability of each of the four possible splitting types
 of the modulo $(\fp)$ reduction: $(3), (2,\ 1), (1,\ 1,\ 1), (1^2,\ 1).$ \smallskip

{\bf Case 5A: $B \notin \pi_\fp\sO_\fp$, and $\overline{f}(x)$ has splitting type $(3)$.}

We count the number of irreducible polynomials in $\F_q[x]$ of the form
$x^3 + Ax + B$. There are $\inv{3}\left(q^3 - q\right)$ monic, irreducible cubic
 polynomials in $\F_q[x]$. We call two monic polynomials $f,g\in\F_q[x]$ equivalent
  if there is a monic linear polynomial $h \in \F_q[x]$ with $f = g \circ h$. These
  equivalence classes are all of size $q$, since there are $q$ choices of $h$
  which all give distinct polynomials. All polynomials in the same equivalence class
   have the same splitting type. There is also a unique polynomial of the form
   $x^3 + Ax + B$ in each equivalence class of cubics. This means that there are
   $\inv{3}(q^2 - 1)$ irreducible cubic polynomials of the form $x^3 + Ax + B$ in $\F_q[x]$.
   So the probability of being in this subcase, given that the polynomial is in
   Case 5 is $\frac{q+1}{3q} = \inv{3} + \inv{3q}$.

   Therefore, this case contributes the following distribution:
\begin{align*}
&\ssy(f) = (3)~~~\mbox{with probability}~~~\left(1- q^{-1}\right)(\inv{3} + \inv{3q}) \\
&\Gal(K_f/K_\fp) = \Z/3\Z~~~\mbox{with probability}~~~\left(1- q^{-1}\right)(\inv{3} + \inv{3q})\\
&\Gal(K_f/K_\fp) = S_3~~~\mbox{with probability}~~~0.
\end{align*}

{\bf Case 5B: $B \notin \pi_\fp\sO_\fp$, and $\overline{f}(x)$ has splitting type $(2,\ 1)$.}

We count the number of polynomials $x^3 + Ax + B \in \F_q[x]$ which have an
 irreducible quadratic factor. There are $\inv{2}\left(q^2 - q\right)$ irreducible monic quadratic polynomials in $\F_q[x]$. For each $x^2 + ax + b \in \F_q[x]$ which is irreducible, the linear polynomial $(x-a)$ is unique such that their product is of the prescribed form.
  Thus, for each monic irreducible quadratic polynomial in $\F_q[x]$ there is a
   unique cubic of the form $x^3 + Ax + B$ which it divides. Thus there are
   $\inv{2}\left(q^2 - q\right)$ polynomials $x^3 + Ax + B \in \F_q[x]$ which have an irreducible quadratic factor. 
   However, whenever $a = 0$ this will give that $B = 0$,
   which is not allowed. There are $\inv{2}(q-1)$ irreducible quadratics of the form
    $x^2 + b$, and so we subtract these from the total count to see that there are
    $\inv{2}\left(q-1\right)^2$ polynomials of the prescribed form which have
    an irreducible quadratic factor. So the probability of being in this subcase,
     given that the polynomial is in Case 5 is $\frac{q-1}{2q} = \inv{2} - \inv{2q}$.

   Therefore, this case contributes the following distribution:
\begin{align*}
&\ssy(f) = (2,\ 1)~~~\mbox{with probability}~~~\left(1- q^{-1}\right)(\inv{2} - \inv{2q}) \\
&\Gal(K_f/K_\fp) = \Z/2\Z~~~\mbox{with probability}~~~\left(1- q^{-1}\right)(\inv{2} - \inv{2q}).
\end{align*}

{\bf Case 5C: $B \notin \pi_\fp\sO_\fp$, and $\overline{f}(x)$ has splitting type $(1,\ 1,\ 1)$.}

We wish to count the number of polynomials $x^3 + Ax + B \in \F_q[x]$ which factor as three distinct linear polynomials. 
These must factor as $(x + a)(x+ b )(x - a - b)$ with
 the following restrictions:
 \[a\neq b,~~~a\neq 0,~~~b\neq 0,~~~b\neq -a,~~~b \neq -2a,~~~ a\neq -2b.\]
 There are $q-1$ possible choices for $a$. Given any $a$, there are $q - 5$ choices
 for $b$ which satisfy the restrictions. Additionally, all six possible re-orderings
 of the triple $(a,b,-a-b)$ produce the same polynomial. So the total number of
 such polynomials is $\frac{(q-1)(q-5)}{6}.$ So the probability of being in this
 subcase, given that the polynomial is in Case 5 is $\frac{q-5}{6q} = \inv{6} - \frac{5}{6q}$.

  Therefore, this case contributes the following distribution:
\begin{align*}
&\ssy(f) = (1,\ 1,\ 1)~~~\mbox{with probability}~~~\left(1- q^{-1}\right)(\inv{6} - \frac{5}{6q}) \\
&\Gal(K_f/K_\fp) = \text{id}~~~\mbox{with probability}~~~\left(1- q^{-1}\right)(\inv{6} - \frac{5}{6q}).
\end{align*}

{\bf Case 5D: $B \notin \pi_\fp\sO_\fp$, and $\overline{f}(x)$ has splitting type $\left(1^2,\ 1\right)$.}

Any polynomial of the form $x^3 + Ax + B \in \F_q[x]$ which has a square factor
 must be of the form $(x+a)^2(x-2a)$, with $a\neq 0$. There are $q-1$ possible choices for $a$, and so the probability of being in this subcase, given that the polynomial is in
 Case 5 is $\frac{1}{q}$.

By Lemma \ref{l:limit} a polynomial in $\sO_\fp[x]$ whose reduction is in
Case 5D has splitting type $(1,\ 1,\ 1)$ with probability $\inv{2q+2}$, and
has splitting type $(2,\ 1)$ with probability $1 - \inv{2q+2}$.

  Therefore, this case contributes the following distribution:
\begin{align*}
&\ssy(f) = (1,\ 1,\ 1)~~~\mbox{with probability}~~~\inv{q}\left(1- q^{-1}\right)\left(\inv{2q+2}\right) \\
&\Gal(K_f/K_\fp) = \text{id}~~~\mbox{with probability}~~~\inv{q}\left(1- q^{-1}\right)\left(\inv{2q+2}\right)\\
&\ssy(f) = (2,\ 1)~~~\mbox{with probability}~~~\inv{q}\left(1- q^{-1}\right)\left(1-\inv{2q+2}\right) \\
&\Gal(K_f/K_\fp) = \Z/2\Z~~~\mbox{with probability}~~~\inv{q}\left(1- q^{-1}\right)\left(1-\inv{2q+2}\right).
\end{align*}

Summing over all cases and dividing by $1 - \inv{q^5}$
we get the stated equalities. For example
\begin{align*}\moneThreek(&\pgpThreek(1,\ 1,\ 1)) = \\ &= \left[\inv{1-q^{-5}}\right]\left[ \inv{2}\left( \frac{1}{q}\left(1- \frac{1}{q}\right)\right) + \left(1- q^{-1}\right)(\inv{6} - \frac{5}{6q}) + \inv{q}\left(1- q^{-1}\right)\left(\inv{2q+2}\right)\right],\end{align*} which are the totals from cases 1, 5C, and 5D, and
\[\moneThreek(\pgpThreek(3)) = \left[\inv{1-q^{-5}}\right]\left[\left(q^{-2}-q^{-3}\right)q^{-2} + \left(q^{-1}-q^{-2}\right)q^{-1} + \left(1- q^{-1}\right)(\inv{3} + \inv{3q})\right],\] which are the totals from cases 3, 4, and 5A.
\end{proof}

\begin{proof}[of Theorem  \ref{t:unramifiedextn3}]

The proof of this theorem is identical to that of Theorem \ref{t:examp3}, except we will only be summing over the cases where the Newton polygon of the polynomial has integral slopes, which is equivalent to the polygon having unramified splitting field. This excludes cases 2, 3, and 4 from the proof of Theorem \ref{t:examp3}. Polynomials in cases 1 and 5 generate unramified extensions. Therefore we have that \[\moneThreek(R_{3,\fp}) =  \left[\inv{1-q^{-5}}\right]\left[\inv{q}\left(1-\inv{q}\right) + 1 - q^{-1}\right],\] which comes from the probability of a polynomial being in cases 1 or 5.

Likewise, we calculate $\moneThreek(R_{3,\fp}(1,\ 1,\ 1))$, $\moneThreek(R_{3,\fp}(2,\ 1))$, and $\moneThreek(R_{3,\fp}(3))$ by summing over cases 1 and 5 and dividing by $1-q^{-5}$. This gives \begin{align*}
&\moneThreek(R_{3,\fp}(1,\ 1,\ 1)) = \left[\inv{1-q^{-5}}\right]\left[\inv{2}\left( \frac{1}{q}\left(1- \frac{1}{q}\right)\right) + \left(1- q^{-1}\right)(\inv{6} - \frac{5}{6q} + \inv{2q^2 + 2q})\right]\\
&\moneThreek(R_{3,\fp}(2,\ 1)) =  \left[\inv{1-q^{-5}}\right]\left[\inv{2}\left( \frac{1}{q}\left(1- \frac{1}{q}\right)\right) + \left(1- q^{-1}\right)(\inv{2} - \inv{2q} + \inv{q}-\inv{2q^2+2q})\right]\\
&\moneThreek(R_{3,\fp}(3)) = \left[\inv{1-q^{-5}}\right]\left[\left(1- q^{-1}\right)(\inv{3} + \inv{3q}) \right].
\end{align*}

\end{proof}

\paragraph{\bf Acknowledgments.}
The author thanks Jeffrey Lagarias for posing these problems, suggesting improvements
of some results, and  assistance editing the paper. 
He   thanks Mike Zieve for a proof of Lemma \ref{l:sqfree},
and  for many editorial comments on  the paper. He thanks Chris Hall for useful conversations on
$p$-adic Galois extensions, and Brian Conrad for asking whether results for $\Q_p$
extend to more general $p$-adic fields.

\end{document}